\documentclass{amsart}
\usepackage{amsmath,amsfonts} 
\usepackage{amsthm} 
\usepackage{amssymb} 
\usepackage[T1]{fontenc} 
\usepackage[utf8]{inputenc}
\usepackage{textcomp}
\usepackage[dvips]{graphicx}
\usepackage{verbatim}
\usepackage{todonotes}
\usepackage{mathrsfs}
\usepackage[english]{babel}
\usepackage[babel]{csquotes}
\usepackage{xy}
\xyoption{all}
\usepackage{hyperref}
\usepackage[style=alphabetic,hyperref,backend=bibtex]{biblatex}
\bibliography{bibliography}

\theoremstyle{definition}
\newtheorem{definition}{Definition}[section]
\newtheorem{question}{Question}
\theoremstyle{plain}
\newtheorem{theorem}[definition]{Theorem}
\newtheorem{proposition}[definition]{Proposition}
\newtheorem{lemma}[definition]{Lemma}

\theoremstyle{remark}
\newtheorem{remark}[definition]{Remark}
\newtheorem{example}[definition]{Example}

\DeclareMathOperator{\Hom}{Hom}
\DeclareMathOperator{\GHA}{GHA}
\DeclareMathOperator{\Spec}{Spec}

\DeclareMathOperator{\Gal}{Gal}
\DeclareMathOperator{\GL}{GL}

\DeclareMathOperator{\Aut}{Aut}
\DeclareMathOperator{\End}{End}

\DeclareMathOperator{\Ha}{H}
\DeclareMathOperator{\id}{id}
\DeclareMathOperator{\HG}{HG}

\newcommand{\Oh}{\mathscr{O}}

\newcommand{\bin}[2]{\genfrac{(}{)}{0pt}{}{#1}{#2}}
\newcommand{\bx}{{\bar{X}}}

\begin{document}
\title{A theory of Galois descent for finite inseparable extensions}
\author{Giulia Battiston}
\date{\today}
\address{Mathematisches Institut, Albert-Ludwigs-Universität, Eckerstr. 1, 79104 Freiburg, Germany}
\thanks{This work was supported by GK1821 ``Cohomological methods in Geometry''}
\subjclass[2010]{14G17, 14A15, 12F15}
\email{giulia.battiston@math.uni-freiburg.de}

\begin{abstract}We present a generalization of Galois descent to finite modular normal field extension $L/K$, using the Heerma--Galois group $\Aut(L[\bx]/K[\bx])$ where $L[\bx]=L[X]/(X^{p^e})$ and $e$ is the exponent of $L$ over $K$.
\end{abstract}
\maketitle

If $A\to B$ is a ring homomorphism, the descent theory of $B$ over $A$ studies when a $B$-algebraic objects is defined over $A$, that is, comes by base change from an $A$-algebraic object.
In case $A=K$ and $B=L$ are fields and $L/K$ is a finite Galois field extension of Galois group $G$, it is a classical result (see for example \cite[Sec.~14.20]{goe}) that algebraic objects over $L$ with a suitable $G$-action are exactly those defined over $K$. 
The goal of this article is to extend this result to finite normal possibly non separable extensions $L/K$: of course $G=\Aut(L/K)$ will not do the job, as $K$ is not in general the fixed field of such a group. Instead, following the ideas of Heerma (see \cite{Hee}), we define the \emph{Heerma--Galois group} to be $\HG(L/K)=\Aut(L[\bx]/K[\bx])$ where $L[\bx]=L[X]/(X^{p^e})$ (and similarly for $K[\bx]$) and where $e$ is such that $L^{p^e}$ is separable over $K$. 

Our main result is then that if $L/K$ is a finite normal modular (see Definition~\ref{modular}) field extension, an algebraic object over $L$ is defined over $K$ if and only if \emph{its base change to $L[\bx]$} admits a suitable $\HG(L/K)$-action (see Theorems~\ref{de-svs}, \ref{de-ss} and \ref{de-vs} for a more precise statement).

\

There exist other possible approaches to a Galois-like theory for inseparable extensions, that may be in principle used to describe the descent theory of $L$ over $K$. The most interesting alternative choices that have been proposed, over time, of an object replacing the Galois group are three: the first one is given by considering  the the group of (bounded rank) higher derivations on $L$ relative to $K$ (see \cite{hd}), the second one concerns the Galois--Hopf algebra $\GHA(L/K)$ as described in \cite{AS} and finally the third one uses the so called automorphism scheme $\underline{\Aut}(L/K)$ (see \cite{beg}), which is a $K$-group scheme representing the functor $T\mapsto\Aut_T(L\times_K T)$ for every $K$-scheme $T$, or its truncated version $\underline{\Aut}_t(L/K)$ as defined by Chase in \cite{ch}.

Let me compare the choice of the  Heerma--Galois with its possible alternatives: first of all it extends the classical Galois theory, namely, when $e=0$ (that is $L/K$ is separable Galois) we find back exactly the Galois group and the theorems of Galois descent. As for the group of higher derivations, it can be seen as a subgroup of the Heerma--Galois group, but the latter is neater to handle and has a more algebraic flavour. Compared to the descent theory in \cite{AS}, the Galois--Hopf algebra approach does not need to base change the $L$-algebraic objects to $L[\bx]$ like we do. It moreover follows quite directly from faithfully flat descent theory, thanks to the fact that $L$ is a $\Spec(\GHA(L/K)^*)$-torsor and the duality between $\GHA(L/K)$-module and $\GHA(L/K)^*$-comodule structures on a $K$-vector space. On the other side the Galois--Hopf algebra acts via the endomorphisms and not via the automorphisms of such an $L$-algebraic object. This may not allow some tricks from Galois descent to work in this more general setting, a good example being the following: if $X$ is defined as the stabilizer in $\GL_n$ of some collection of sub-vector spaces $\{W_i\}$ of $L^n$, and $G$ is acting via automorphisms, for every $\sigma\in G$ one has that $\sigma(X)$ is the stabilizer of $\{\sigma(W_i)\}$, hence the problem of descending $X$ can be reduced to the invariance of  $\{W_i\}$ under the action of $G$. Of course, this does not apply if we are considering the action of a Hopf algebra via endomorphisms.
All in all, as $L[\bx]$ has still many desirable properties for a base ring (it is the infinitesimal deformation space of a field and an auto-injective ring, many of the schemes defined via functor of points are still representable, and so on), the loss of a field as base scheme seems little with respect to the gain of an action through the automorphisms, that potentially allows to extend many of the applications of separable Galois descent to the inseparable case.
As for the third alternative, the automorphism scheme is naturally related with the Heerma--Galois group as
\[\HG(L/K)=\underline{\Aut}(L/K)(K[\bx])=\underline{\Aut}_t(L/K)(K[\bx]).\]
It was already known that if $L/K$ is normal and modular, an action of $\Aut(L/K)$ induces an action of $\GHA(L/K)^*$ (in fact, of any group scheme under which $L/K$ is a torsor, see Theorems~\ref{torsor} and \ref{factor}), hence, by the theory of descent along torsors, a descent data. Our main result proves that in fact it is enough to have an action of its $K[\bx]$-points in order to have descent over $K$.

The article is divided as follow: the first section is devoted to define the Heerma--Galois group and to explore its relation with the group of higher order differential operators. The second one applies the results of \cite{AS} in order to obtain descent conditions on $L$-vector spaces and consequently on algebras and more in general separated schemes of finite type. In the third section we explain the connections between the Heerma--Galois group and the (truncated) automorphism scheme. The fourth section is devoted to a little generalization of the results of the second one, namely proving that it is enough to have a $\HG$-action on a (possibly non trivial) infinitesimal deformation of the algebraic object we are interested in. Finally, in the last section we collect some natural questions that still need to be inquired.

\subsection*{Acknowledgements} I would like to thank Lenny Taelman for pointing out to me the existence of the automorphism scheme and Moshe Kamensky for providing the reference of Pillay and Poizat. I am indebted to the notes of Milne (\cite[Ch.~16]{mil}) who have widely influenced the structure of the section on descent theory.
\subsection*{Notation} If $A\to B$ is a ring homomorphism and $X$ is a scheme over $\Spec A$, we use the notation $X\otimes_AB$ to abbreviate $X\times_{\Spec A}\Spec B$. By an \emph{algebraic object} $R$ over a ring $A$ we always mean that $R$ is an $A$-module with some additional algebraic structure (for example a ring, a Hopf algebra, a Lie algebra and so on).

\section{The Heerma--Galois group}

Let us fix $K$ a field of positive characteristic, and an algebraic closure $\bar{K}$ of $K$. Let $\alpha\in\bar{K}$ be an algebraic element which is inseparable over $K$. The elementary extension $K(\alpha)$ is the simpler example of (finite purely inseparable) modular extensions of finite exponent: 

\begin{definition}[{ see \cite[Thm.~1]{Swee} and \cite[def~1.1]{devmo}}]\label{modular}
An algebraic extension $L$ of $K$, is said to be of \emph{finite exponent} if there is a positive natural number $n$ such that $L^{p^n}$ is separable over $K$. The minimal of such $n$ is said to be the \emph{exponent} of $L$ over $K$.

An algebraic extension $L$ of $K$ is \emph{modular} if it is isomorphic to a (possibly infinite) tensor product over $K$ of elementary extensions (that is, of extensions of $K$ generated by one element).
\end{definition}

We will discuss Galois descent when $L/K$ is a finite modular normal extension. Note though that this condition can be achieved in many situations as if $L/K$ is purely inseparable, or finite and normal, then by \cite[Thm.~6]{Mor} there exists a \emph{modular closure of $L$ with respect to $K$} which is the smallest field containing $L$ which is modular over $K$. Moreover the modular closure of $L$ over $K$ is finite and purely inseparable over $L$ (in particular it is finite over $K$ whenever $L$ is) and its exponent over $K$ is equal to the exponent of $L$ over $K$.

\

For the rest of this article, $L$ will always be a finite normal modular extension of $K$. In this setting, Heerema in \cite{Hee} constructed a Galois group theory as follows: let us denote by $e\in\mathbb{N}$ the exponent of $L$ over $K$ and let $L[\bx]=L[X]/(X^{p^e})$
\footnote{Note that in \cite{Hee} $L[\bx]$ is defined as the truncated polynomial $L[X]$ modulo $X^{p^{e-1}+1}$ rather than $X^{p^e}$. It is straightforward to check that this does not make any difference for our results, and the latter seems a better choice for this article.}
. Let 
\begin{equation}\label{A}
A=\{\phi\in\Aut(L[\bx])\mid \phi(\bx)=\bx\},
\end{equation} 
where $\Aut(L[\bx])$ is the group of ring automorphisms of $L[\bx]$ and for any $G\subset A$ subgroup, let us denote 
\[L^G=\{a\in L\mid \alpha(a)=a \;\forall\alpha\in G\}.\]
Then by \cite[Thm.~3.1]{Hee}, there exists a subgroup $G\subset A$ so that $K=L^G$. As $G$ fixes $K$ and $\bx$, we can simply consider $G$ to be $\Aut(L[\bx]/K[\bx])$, we will call it the \emph{Heerma--Galois group of $L$ over $K$} and denote it with $\HG(L/K)$ or simply $\HG$ when $L$ and $K$ are fixed.

\begin{remark}
Unlike the classical Galois theory, it is not true that the $L[\bx]^{\HG}=K[\bx]$, as for example $a_0+a_{p^{e-1}}\bx^{p^{e-1}}$ is fixed by $\HG$ for every $a_{p^{e-1}}\in L$ and $a_0\in K$.
\end{remark}

\begin{definition}
A \emph{higher derivation of rank $n$} on $L$ is a family of additive maps $\{d^{(k)}:L\to L\}$ for $0\leq k\leq n$ such that
\begin{itemize}
\item[i)] $d^{(0)}=\id$
\item[ii)] for every $0\leq k<n$ and $a,b\in L$ one has
\[d^{(k)}(ab)=\sum_{i=0}^k d^{(i)}(a)d^{(k-i)}(b).\]
\end{itemize}
A higher derivation $\{d^{(k)}\}$ is \emph{relative to $K$} if every $d^{(k)}$ is a $K$-linear map (which by $(ii)$ is equivalent to ask that $d^{(k)}(a)=0$ for every $a\in K$ and every $k>0$). Let us denote by $\mathcal{H}^{\leq n}(L/K)$ the set of higher derivations of rank $n$ on $L$ relative to $K$. It admits a group structure defined as $\{d^{(k)}\}\cdot\{e^{(k)}\}=\{f^{(k)}\}$ with $f^{(k)}=\sum_{i=0}^k d^{(i)}e^{(k-i)}$.
\end{definition}

For our purposes, let us unravel the main ingredient in the proof of \cite[Thm.~3.1]{Hee}. It is a classical result that if $L$ is modular over $K$ of exponent $e$, then $K$ is the field of constants of one higher derivation. That is, there exists $\{d^{(k)}\}$ higher derivation on $L$ such that $K=\{a\in L\mid d^{(k)}(a)=0 \; \forall k>0\}$. Moreover, this higher derivation can be taken to be of rank $p^e-1$. Actually the two properties, of being modular of exponent $e$ and of being the field of constants of one (and hence of all) higher derivations of rank $p^e-1$  on $L$ relative to $K$ are equivalent by \cite[Thm.~1]{Swee}. The  core of the Galois correspondence in \cite{Hee} sits then in the fact that the morphism
\begin{equation}\label{delta}
\delta\mathcal{H}^{\leq p^e-1}(L/K)\to \HG(L/K)
\end{equation}
defined as $\delta(\{d^{(k)}\})(\bx)=\bx$ and $\delta(\{d^{(k)}\})(a)=\sum_{k=0}^{p^e-1} d^{(k)}(a)\bx^k$ for $a\in L$ is an isomorphism onto \[A_0=\{\phi\in A\mid \phi(a)-a\in (\bx)\;\forall a\in L\},\]
thus allowing to see higher derivations of rank $p^e-1$ as a subgroup of $\HG(L/K)$.

\section{Descending objects from $L$ to $K$}\label{sec:desc} 

If $V$ is a $K$-vector space, and $W$ is a sub-vector space of $V_L=V\otimes_KL$, it is natural to ask whether $W$ is \emph{defined over $K$}, namely whether there exists $W_0\subset V$ a $K$-sub-vector space such that $W=W_0\otimes_KL$. More generally, given a $L$-vector space $W$ we say that $W_0\subset W$ is a \emph{$K$-form} if $W_0$ has a structure of $K$-vector space and the morphism $W_0\otimes_KL\to W$ given by $w\otimes a\mapsto a\cdot w$ is an isomorphism of $L$-vector spaces. 

Of course, every $L$-vector space admits a $K$-form, but this is not the case if we endow $W$ with an additional algebraic structure: for example if $W$ is an $L$-algebra, we say that $W_0\subset W$ is a \emph{$K$-form} for $W$ (as an algebra) if $W_0$ has a structure of $K$-algebra and the morphism $W_0\otimes_KL$ is an isomorphism of $L$-algebras. Indeed, a $K$-form of an algebra is also a $K$-form of the underlying vector space, but the converse does not hold. 
\begin{remark}
Another interesting property that one may wish to descend from $L$ to $K$, as suggested in the beginning of the paragraph, is that of an \emph{embedded} $K$-form, namely we want not only to descend an object, in our example $W$, but also its embedding in a bigger object that is defined over $K$, in our example $W\subset V\otimes_KL$. More in general one is interested in descending morphisms $\phi:V\otimes_KL\to V'\otimes_KL$ between objects that are defined over $K$, namely understanding whether there is a morphism $\phi_0:V\to V'$ such that $\phi=\phi_0\otimes id$.
\end{remark}

\subsection{Descent of sub-vector spaces}
We will focus first on the descent of vector spaces without additional algebraic structure. Before stating the theorem that will be central in proving our main result, let us introduce some notation. 
Let $L/K$ be a finite modular normal extension, let us once and for all fix a decomposition  
\begin{equation} \label{dec} 
L=K(\alpha_0)\otimes_KK(\alpha_1)\otimes_K\dotsb\otimes_KK(\alpha_m)
\end{equation}
with $\alpha_0$ separable over $K$ and $\alpha_i$ a $p^{n_i}$-th root of $a_i\in K$, for $i=1,\dotsc,m$. Let $G$ be the Galois group of $K(\alpha_0)$ over $K$, then we define $\Ha_0=K[G]$ and $H_i$ to be the Hopf algebra over $K$ defined as follows: the generators are $D_i^{(k)}$ with $k=1,\dotsc, p^{n_i}-1$, counit $\epsilon: \Ha_i\to K$ is given by $D_i^{(k)}=0$ if $k>0$ and $D_i^{(0)}=1$, multiplication given by 
\[D_i^{(h)}D_i^{(k)}=\bin{h+k}{h} D_i^{(h+k)}\] for $s+t<p^{n_i}$ and zero otherwise and comultiplication given by \[\Delta(D_i^{(k)})=\sum_{h=0}^kD_i^{(h)}\otimes D_i^{(k-h)}.\]

Let us finally define the \emph{Galois-Hopf algebra of $L$ over $K$} to be 
\[\GHA(L/K)=\Ha_0\otimes_K \Ha_1 \otimes_K\dotsb\otimes_K \Ha_m,\]
 then $\GHA(L/K)$ acts on $L$ via  
\[D_i^{(k)}(\alpha_i^l)=\bin{l}{k}\alpha_i^{l-k} \text{ and } D_g(\alpha_0)=g(\alpha_0),\]
where $D_g$, for $g\in G$, denote the generators of $K[G]$.
 Note that if $V$ is a $K$-vector space, the action of $\GHA(L/K)$ on $L$ that we just described induces a natural action on $V_L=V\otimes_KL$  simply by tensoring with the trivial action on $V$, and this action is $L$-semilinear:

\begin{definition}
Let $V$ be an $L$-vector space, then an action of $\GHA(L/K)$ is called \emph{$L$-semilinear} if $V$ is an $\GHA(L/K)$-module under this action and for every $D\in\GHA(L/K)$, $a\in L$ and $v\in V$ one has
\[D(a\cdot v)=\Delta(D)(a\otimes v).\]
\end{definition}

We have now all the definitions that we need in order to state the key theorem that relies $K$-forms and $\GHA(L/K)$-semilinear actions:

\begin{theorem}[{\cite[Lemma 1.2,3; Thm.~1.2,5; Thm.1.2,8]{AS}}] \label{hopf}Let $L$ be a normal modular finite field extension of $K$ and let $V$ a $L$-vector space on which $\GHA=\GHA(L/K)$ acts $L$-semilinearly. Let $\epsilon$ be the counit of $\GHA$, then if we define
\[V^{\GHA}=\{v\in V\mid D \cdot v=\epsilon(D)\cdot v, \; \forall D\in\GHA\},\]
$V^{\GHA}$ is a $K$-form for $V$.
\end{theorem}

\begin{remark}
As noted in the introduction, the previous theorem is a direct application of faithfully flat descent theory, once remarked that $L$ is a torsor for $\GHA(L/K)^*$, the Cartan dual of $\GHA(L/K)$, hence a module structure of the latter on $V$ gives a comodule structure of the first and thus, by descent along torsors (see for example \cite[Sec.~14.21]{goe}), a $K$-form on $V$.
\end{remark}

In particular, if $V$ is a vector space over $K$ and we consider $V_L=L\otimes_K V$ with the natural $\GHA(L/K)$-action, we have that the subspaces of $V_L$ are defined over $K$ are exactly those invariant under the $\GHA(L/K)$-action.
We want now to translate this in a more Galois-like setting, namely proving the following

\begin{theorem}\label{de-svs}
Let $V$ be a vector space over $K$ and $L$ a finite modular normal extension, let $\HG(L/K)$ be the Heerma--Galois group of $L$ over $K$. Let $W$ be a sub-vector space of $V_L=L\otimes_k V$, then the following are equivalent:
\begin{itemize}
\item[i)] there exists a $K$-sub-vector space $W_0\subset V$ such that $W=L\otimes_k W_0$;
\item[ii)] $W$ is stable under the natural $\GHA(L/K)$-action over $V_L$;
\item[iii)] $W\otimes_LL[\bx]$ is stable under the natural $\HG(L/K)$-action on 
\[V_{L[\bx]}=V\otimes_L L[\bx]=V\otimes_KK[\bx].\]
\end{itemize}
\end{theorem}
\begin{proof}
If $(i)$ holds, then clearly $(ii)$ and $(iii)$ hold as well. Moreover by Theorem \ref{hopf}, $(ii)$ implies $(i)$.
We are left to prove that $(iii)$ implies $(ii)$. Let us first fix some notation: recall that we have fixed a decomposition \ref{dec} of $L$. Let $D_g$, $g\in G$ be the generators of the Hopf algebra $\Ha_0=K[G]$ and let $\phi_g\in\HG(L/K)$ be defined as $g\otimes \id\otimes\dotsm\otimes \id$ on $L$ and $\phi_g(\bx)=\bx$. Let us moreover define the following higher derivations of rank $p^e-1$, denoted $\{d_i^{(k)}\}$, $i=1,\dotsc,m$ defined by $d_i^{(0)}=\id$ and 
\[
d_i^{(k)}=\begin{cases} D_{\id} \otimes D_1^{(0)} \otimes\dotsc D_{i-1}^{(0)}\otimes D_i^{(\frac{k}{p^{e-n_i}})}\otimes D_{i+1}^{(0)}\dotsc\otimes D_m^{(0)} & \text{ if } p^{e-n_i}\mid k  \\
0 & \text{ otherwise}\\
\end{cases}\]
(in short, we are shifting the action of the $D_i^s$s on $L$ so that they form a higher derivation of rank $p^e-1$ rather than $p^{n_i}-1$). 
Now that we have fixed the notation, we need to show that $W$ is stable under the action of the $D_g$s, for $g\in G$ and of the $D_i^{(k)}$s, for $i=1,\dotsc, m$, $k=1,\dotsc, p^{(n_i-1)}$ (by abuse of notation, we still denote by $D_g$ what should be denoted $D_g\otimes D_1^{(0)}\otimes\dotsb\otimes D_m^{(0)}$ and similarly for $D_i^{(k)}$).
Let $\phi_i=\delta(\{d_i^{(k)}\})\in\HG(L/K)$, and let us fix a basis $v_j$ of $V$ over $K$, a basis $w_s$ of $W$ over $L$. We will denote still by $v_j$ the $L[\bx]$-basis $v_j\otimes 1$ of $V\otimes_KL[\bx]$, and similarly for $w_s$. Let $w\in W$, $w=\sum_j \gamma_jv_j$, as $W\otimes_LL[\bx]$ is invariant under the action of $\HG(L/K)$ and the basis $v_j$ is stable under the action of $\HG(L/K)$, then 
\[\begin{split}
\phi_i(w)&=\sum_j\phi_i(\gamma_j)\phi_i(v_j)\\
&=\sum_j \phi_i(\gamma_j)v_j=\sum_k\sum_j D_i^{(k)}(\gamma_j)\bx^{kp^{e-n_i}} v_j\\
&=\sum_i\sum_s\beta_s^i\bx^iw_s
\end{split}\]
for some $\beta_s^i\in L$.

As $v_j\bx^i$ is a basis for $V\cdot\bx^i$ where $V\otimes L[\bx]=V\cdot 1\oplus V\cdot \bx\oplus \dotsc \oplus V\cdot \bx^{p^e-1}$, it follows that $D_i^{(k)}(w)=\sum_j D_i^{(k)}(\gamma_j)v_j=\sum_s\beta_s^{kp^{e-n_i}}w_s$. In particular, $W$ is stable under the action of the $D_i^{(k)}$s.
As for the $D_g$s, the proof goes similarly considering $\phi_g\in G$ defined as $\phi_g(a)=g(a)$ for $a\in L$ and $\phi_g(\bx)=\bx$.
\end{proof}

\begin{remark}
The previous proof needs a rather involved notation because $n_i$ may differ from $e$ for some $i$ and thus one needs to shift the $D_i^{(k)}$: in order to understand why, let $L=K(\alpha)$ with $\alpha$ $p$-th root of $a\in K$, that is let $e=1$. Let $\{d^{(k)}\}$ a higher derivation of rank at least $p$, then by Lucas' theorem $0=d^{(p)}(a)=d^{(p)}(\alpha^p)=(d^{(1)}(\alpha))^p$. Hence, every higher derivation of rank at least $p$ must have $d^{(1)}=0$, thus there is no way to define a $\{d^{(k)}\}\in\mathcal{H}^{p^2-1}$ such that $d^{(1)}=D^{(1)}$.
\end{remark}

\subsection{Descending sub-schemes and morphisms}

Once understood the descent of sub-vector spaces, the descent of sub-schemes  follows directly from the fact that the closure of a sub-vector space of a ring $R$ under multiplication, that is the property of being an ideal, can be checked after base change:

\begin{lemma}\label{de-ss}
Let $X$ be a separated scheme over $K$, $L$ a finite normal modular extension of $K$ and $Z$ a closed sub-scheme of $X_L=X\otimes_KL$. Then there exists a closed sub-scheme $Z_0\subset X$ such that $Z=Z_0\otimes_KL$ if and only if $Z\otimes_LL[\bx]$ is stable under the natural action of $\HG(L/K)$ over $X_L\otimes_LL[\bx]$.
\end{lemma}

\begin{proof} If $Z$ descends to $K$ then it is clearly invariant under $\HG(L/K)$. 
As $X$ is separated, we can cover it with affine opens $\{U_i\}$ whose intersection are affine and hence reduce ourselves to the case where $X$ is affine with global sections $R$. Let $I$ be the defining ideal of $Z$ inside $R\otimes_KL$, then as a vector space it is $\HG(L/K)$-stable, as $Z$ is, hence by Theorem \ref{de-svs} there exists a sub-vector space $I_0$ of $R$ such that $I=I_0\otimes_KL$ as vector spaces. But $I_0$ is an ideal if (and only if) $I$ is one: the only property to check is the closure of $I_0$ under multiplication by an element $r\in R$, and it is enough to do it after extension of scalars.
\end{proof} 

\begin{remark}
Note that if $L/K$ is purely inseparable, then $\Spec L\to \Spec K$ is a universal homeomorphism, in particular the topological spaces underlying $X$ and $X\otimes_KL$ are homeomorphic for every $K$-scheme $X$. Hence the underlying topological space of every sub-scheme of $X\otimes_KL$ is ``defined'' over $K$ but so is not its algebraic structure.
\end{remark}

The descent of a morphism between separated schemes follows now from the previous case using the graph of such a morphism:
\begin{lemma}
Let $X$ and $Y$ two separated schemes over $K$, let $L$ be a modular extension of $K$. Then the image of the map
\[\circ_L:\Hom_K(X,Y)\to\Hom_L(X_L,Y_L)\]
sending $f$ to $f\otimes \id$ consists of all morphisms $g\in\Hom_L(X_L,Y_L)$ such that \[g\otimes \id:X_L\otimes_L L[\bx]\to Y_L\otimes_L L[\bx]\] is $\HG(L/K)$-equivariant.

\end{lemma}
\begin{proof} If $g$ is in the image of $\circ_L$, that is $g=f\otimes \id$ for some $f\in\Hom_K(X,Y)$ then $g\otimes \id$ is certainly $\HG$-equivariant. On the other hand, let $\Gamma_g$ be the graph of $g$, namely let $\Gamma_g:X_L\to X_L\times_{\Spec L} Y_L$  be the immersion given by $(id,g)$. Then $g=\pi_2\circ \Gamma_g$, and as $\pi_2$ is defined over $K$, it is enough to understand when $\Gamma_g$ descends to $K$. But as we are dealing with separated schemes, all graphs are closed immersion, so we are reduced to Lemma \ref{de-ss}.
\end{proof}

\subsection{Descending vector spaces}
If $L/K$ is a finite Galois extension, the classical Galois descent theory says that any $\sigma$-linear $\Gal(L/K)$-action on a $L$-vector space $V$ induces a $K$-form (in fact, it is a one on one correspondence thanks to Hilbert's Theorem 90).

As for sub-vector spaces, we would like to use Theorem \ref{hopf} in order to get a similar condition for modular extensions. Unlike the classical Galois descent case, though, Theorem \ref{hopf} does not gives a correspondence between $\GHA(L/K)$-semilinear actions on $V$ and $K$-forms, even though a description of $K$-forms can be given in terms of precocycles (see \cite[Thm.~1.2,8]{AS}) or of cocycle, if seen via faithfully flat descent (see \cite[Sec.~4]{des}). We will not be interested in describing the set of $K$-form, but rather deciding whether a $K$-form exists. Of course this is a trivial problem for vector spaces, but will be important when we endow the vector space with an algebraic structure.

\begin{definition} Let $V$ be a $L[\bx]$-module, we say that an action of $\HG(L/K)$ on $V$ is $\sigma$-linear if for every $a\in L[\bx]$ and $v,v'\in V$, and for every $\sigma\in\HG(L/K)$ one has
\[\sigma(v+v')=\sigma(v)+\sigma(v') \text{ and } \sigma(a\cdot v)=\sigma(a)\cdot \sigma(v).\]
\end{definition}

\begin{theorem}\label{de-vs}
Let $L$ be a modular normal finite extension over $K$ and let $V$ be a $L$-vector space endowed with a $\sigma$-linear action of $\HG(L/K)$ on $V\otimes_LL[\bx]$. Then the action of $\HG(L/K)$ induces a $L$-semilinear action of $\GHA(L/K)$ on $V$, in particular, $V$ admits a $K$-form.
\end{theorem}

\begin{proof}
The last statement follows from Theorem \ref{de-svs}. As for the first claim, let $e$ be the exponent of $L$ over $K$, and let us consider the decomposition of $L$ as in \eqref{dec}. 

 Fix a basis $v_s$ of $V$ over $L$, and let $\phi_i$ and $\phi_g$ be as in Thm.~\ref{de-svs}. Let $v$ be one of the $v_s$ and $1\leq i\leq m$, $0\leq k< p^{n_i}$, then we can write $\phi_i(v)=\sum_j\sum_s \gamma_j^sv_s \bx^j$, where by abuse of notation we use $v$ and $v_s$ instead of $v\otimes 1$ and $v_s\otimes 1$, respectively. Let us define 
 \begin{equation}\label{un}D_i^{(k)}(v)=\sum_s \gamma_{kp^{e-n_i}}^sv_s.
 \end{equation}
 Similarly, let us define
 \begin{equation}\label{du} D_g(v)=\sum_s \theta_0^sv_s,
\end{equation}
where $\phi_g(v)=\sum_j\sum_s\theta_j^sv_s\bx^j$.
It is evident that the action is additive and $K[\bx]$-linear, we have to check that this defines a semilinear $\GHA(L/K)$-action: namely that $V$ is an $\GHA(L/K)$-module under this action and that for every $D\in\GHA(L/K)$, $a\in L$ and $v\in V$ one has
\[D(a\cdot v)=\Delta(D)(a\otimes v).\]
Let us check first the second property: by additivity it is enough to do prove it for an element of the basis $v$ and for the $D_i^{(k)}$s and the $D_g$s, as they generate $\GHA(L/K)$. Let $\lambda\in L$, then 
\[\begin{split}
\phi_i(\lambda\cdot v)=&\phi_i(\lambda)\phi_i(v)\\
=&\big[\lambda+D^{(1)}(\lambda)\bx^{p^{e-n_i}}+\dotsb\big]\cdot\big[\sum_s \gamma_0^sv_s+\sum_s\gamma_{p^{e-n_i}}^sv_s\bx^{p^{e-n_i}}+\dotsb\big]\\
=&\lambda\cdot\sum_s\gamma_0^sv_s+[\sum_{j=0}^1\sum_sD^{(j)}(\lambda)\gamma_{(1-j)p^{e-n_i}}^sv_s]\bx^{p^{e-n_i}}+\dotsb
\end{split}
\]
and the semilinearity follows from the definition of the action and the fact that
\[\Delta(D_i^{(k)})=\sum_{h=0}^kD_i^{(h)}\otimes D_i^{(k-h)}.\]
 We can then use the very same argument with $\phi_g$ and $D_g$ instead of $\phi_i$ and $D_i^{(k)}$.
We are left to check that \eqref{un} and \eqref{du} define a $\GHA(L/K)$-module structure on $V$. To do so, it is enough to check that the relations between the $D_i^{(k)}$s and $D_g$s are satisfied by their image in $\End_K(V)$.
These relations are the following: for every $g\in G$, every $1\leq i,j\leq m$ and every $h\leq p^{n_i}$, $k\leq p^{n_j}$ one has
\[\begin{split}
[D_i^{(h)},D_g]&=0\\
D_gD_h&=D_{gh}\\
D_i^{(h)}D_j^{(k)}&=
\begin{cases}
\bin{h+k}{k} D_i^{(h+k)}&\text{if } h+k<p^{n_i} \text{ and } i=j\\
0&\text{else.}
\end{cases}
\end{split}
\]
for $s+t<p^{n_i}$ and zero otherwise. It is easy to see they follow from the fact that similar relations hold between $\phi_i$, $\phi_j$ and $\phi_g$ in $\HG(L/K)$ (and hence for their images in $End_{K[\bx]}(V\otimes_KK[\bx])$), together with the semilinearity we just proved.
\end{proof}

\subsection{Descending schemes}
Let as before $L/K$ be a finite normal modular extension, we can now use the results of the previous section to determine when an $L$-algebra admits a $K$-form, as an algebra.

\begin{definition} Let $R$ be a $L[\bx]$-algebra, we say that an action of $\HG(L/K)$ on $R$ is $\sigma$-linear if it for every $a,b\in R$ and for every $\sigma\in\HG(L/K)$ one has
\[\sigma(a+b)=\sigma(a)+\sigma(b)\text{ and }\sigma(a\cdot b)=\sigma(a)\cdot \sigma(b).\]
\end{definition}

Note that a $\sigma$-linear action on $R$ does not respect the $L[\bx]$-algebra structure but it does respect the induced $K$ and $K[\bx]$-algebra structure.

\begin{definition}
Let $R$ be an $L$-algebra, then an action of $\GHA(L/K)$ is called \emph{$L$-semilinear} if $R$ is an $\GHA(L/K)$-module for this action and for every  $a,b\in R$ and every $D\in\GHA(L/K)$ one has
\[D(a\cdot b)=\Delta(D)(a\otimes b).\]
\end{definition}

\begin{remark}We are here focusing on the category of schemes, but once the descent of vector spaces is established, descent results are also true for other kind of algebraic objects, such as Lie algebras or modules. Of course, we need to assume that the $\HG(L/K)$-action respects the algebraic structure we are interested in.
\end{remark}

\begin{theorem}\label{de-al}
Let $L$ be a modular finite extension over $K$ and let $R$ be a $L$-algebra endowed with a $\sigma$-linear action of $\HG(L/K)$ on the $L[\bx]$-algebra $R\otimes_LL[\bx]$. Then the action of $\HG(L/K)$ induces a $L$-semilinear action of $\GHA(L/K)$ on $R$, in particular, $R$ admits a $K$-form as an $L$-algebra.
\end{theorem}
\begin{proof}
The proof of the first claim goes similarly to Theorem~\ref{de-vs}, we only need to prove that if the action of $\GHA=\GHA(L/K)$ is $L$-semilinear on the $L$-algebra then 
\[R^{\GHA}=\{r\in R\mid D\cdot r=\epsilon(D)\cdot r, \;\forall D\in \GHA\}\]
 has an induced structure of $K$-algebra: by linearity it is enough to check the action on elements of the form $D_g \otimes D_1^{(k_1)} \otimes\dotsc\otimes D_m^{(k_m)}$, and thus (by the relation between multiplication and comultiplication on an Hopf algebra) on the $D_i^{(k)}$ and $D_g$ (with the same abuse of notation as in the proof of Theorem~\ref{de-svs}). On these element, though, this is clear due to the very definition of the antipode on the $\Ha_i$.
Hence, $R^{\GHA}$ is closed under multiplication (and is a $K$-algebra), so it follows that it is a $K$-form for the $L$-algebra $R$. Note that one could derive this last point from the general theory of \cite{AS} on the descent of algebraic structures but one needs to go through quite some notation in order to unravel it.
\end{proof}

Let finally $X$ be any separated scheme. Mimicking the Galois case, let $X_{L[\bx],\sigma}$ be the base change of $X$ to $L[\bx]$ via $\sigma\circ\iota$, with $\iota:L\hookrightarrow L[\bx]$ is the structure morphism.
\begin{definition} A \emph{$K$-descent data on $X$} is a collection of isomorphisms 
\[\phi_\sigma:X_{L[\bx],\sigma}\to X_{L[\bx],\id}\] for  $\sigma\in \HG(L/K)$, such that $\phi_\sigma\cdot \sigma^*(\phi_\tau)=\phi_{\sigma\tau}$.
\end{definition}

Then we have the following:

\begin{theorem}
Let $L/K$ be a modular normal finite field extension and let $X$ a separated $L$-scheme. Then $X$ is defined over $K$ (that is, there exists $X_0$ a $K$-scheme such that $X\simeq X_0\otimes_KL$ as $L$-schemes) if and only if there is a $K$-descent data on $X$ and an affine covering $\{U_i\}_{i\in I}$ of $X$ such that $\phi_\sigma(U_{L[\bx],\sigma})=U_{L[\bx],\id}$ for every $U\in\{U_i\}_{i\in I}$ and every $\sigma\in\HG(L/K)$.
\end{theorem} 

\begin{proof}
As $X$ is separated, $U\cap V$ is affine and moreover
\[\psi_{U,V}:\Oh(U)\otimes_L\Oh(V)\mapsto \Oh(U\cap V)\]
is surjective for every $U,V\in\{U_i\}_{i\in I}$. As $\phi_\sigma(U_{L[\bx],\sigma})=U_{L[\bx],\id}$, the descent data induces an action of $\HG(L/K)$ on $\Oh(U)$ for every $U\in\{U_i\}_{i\in I}$, hence, by Theorem~\ref{de-al}, a descent $U_0$ of $U$ for every $U\in\{U_i\}_{i\in I}$. As the $\phi_\sigma$ are global morphisms, this $\HG(L/K)$-action is compatible on the intersections, that is for every $U,V\in\{U_i\}_{i\in I}$, the morphism $\psi_{U,V}$ induces an $HG(L/K)$-action on $\Oh(U\cap V)$ which gives the gluing for $U_0$ and $V_0$.
\end{proof}

\section{The automorphism scheme}\label{ffd}

The goal of this section is to understand the role of the (truncated) automorphism group scheme in the framework of descent along modular field extensions and its connection with the Heerma--Galois group. Let $A$ be an algebraic object over $R$, that is an $R$-module endowed with some algebraic structure (for example a ring structure), and let $\Aut_R(A)$ be the subgroup of the automorphisms of $A$ as a module that respect the additional algebraic structure. Hence one can define the group valued functor $\underline{\Aut}(A/R)$ on every $A$-algebra $A'$ as
\[A'\mapsto \Aut_{A'}(R\otimes_A A').\]
We denote again by $\underline{\Aut}(A/R)$ the group scheme representing this functor, when it is representable.
If $L/K$ is a finite field extension, $\underline{\Aut}(L/K)$ is called the \emph{automorphism scheme of $L/K$} and was introduced by B\'egueri in \cite{beg}.

Recall that a \emph{truncated  $K$-scheme} is an affine schemes with global sections of the form $K[t_1,\dotsc,t_m]/(t_1^{n_1},\dotsc,t_m^{n_m})$ for some $m, n_i\in\mathbb{N}$. If $L/K$ is finite  by restricting ourselves to the category of truncated schemes we obtain the \emph{truncated automorphism scheme of $L/K$}, denoted by $\underline{\Aut}_t(L/K)$, as defined by Chase in \cite{ch}. Namely this $K$-group scheme represents the functor
\[T\mapsto \Aut_T(T\otimes_KL)\]
 in the category of truncated $K$-schemes. In particular, $\underline{\Aut}_t(L/K)$ is itself a truncated $K$-group scheme and it is a group sub-scheme of $\underline{\Aut}(L/K)$.

In case $L/K$ is Galois, the descent of algebraic objects follows quite directly from faithfully flat descent, thanks to the fact that if $L/K$ is Galois of group $G$, then $L$ is a $\underline{G}$-torsor, where $\underline{G}=\underline{\Aut}(L/K)$ denotes the constant $K$-group scheme associated to $G$. Descent along torsors (see for example \cite[Sec.~14.21]{goe}) implies then that if $M$ is a $L$-module, it is equivalent to give a descent data on $M$ and a $G$-action on $M$ (or, equivalently, a $K[G]^*$-comodule structure, where $K[G]^*$ is the Cartier dual of the group algebra $K[G]$). It is an easy fact, that $L/K$ is Galois of group $G$ if and only if it is a torsor under a constant group scheme, which then can only be $\underline{G}$.

If, on the other hand, we are dealing with purely inseparable extensions, the following holds:
\begin{lemma}[{\cite[Prop.~5.2]{ch}}]\label{torsor}
A finite field extension $L/K$ is normal and purely inseparable if and only if 	$L$ is a torsor under some truncated $K$-group scheme. 
\end{lemma}

\begin{example}
For example $L/K$ is a torsor under $\GHA(L/K)^*$, which is indeed a truncated $K$-scheme if $L/K$ is purely inseparable. 
\end{example}

In the inseparable case, though, the situation is a little more involved than in the separable one, as there can be non isomorphic group schemes under which $L$ is a torsor, nevertheless there is a universal object under which all of these actions actions must factor: 
\begin{theorem}[{\cite[2.1(a)]{ch}}]\label{factor}
Let $L/K$ be a finite field extension, and $H$  a truncated $K$-group scheme acting on $L$. Then there exists a unique morphism \[\gamma_H:H\to \underline{\Aut}_t(L/K)\]of group schemes preserving the action of $H$ on $L$.
\end{theorem}

In particular, if $L/K$ is a finite normal modular extension let us write $L=K(\alpha)\otimes _KL'$, with $K(\alpha)$ Galois over $K$ with Galois group $G$ and $L'$ purely inseparable over $K$. Let $R$ be an algebraic object over $L$, if $R\simeq R_0\otimes_K L$ for some $R_0$ algebraic object over $K$, then there is a natural transformation of group functors
\begin{equation}\label{nat}
\eta^R:\underline{\Aut}(L/K)\to\underline{\Aut}(R/K)
\end{equation}
whose restriction to $\underline{G}\times\underline{\Aut}_t(L'/K)<\underline{\Aut}(L/K)$ will be denoted by $\eta_t^R$. On the other hand if we have such a natural transformation, and $H$ is as in Theorem~\ref{torsor} so that $L'/K$ is an $H$-torsor, then $L$ is a $\underline{G}\times H$-torsor, and and the pullback of $\eta_t^R$ via $\id\otimes \gamma_H$ (see Theorem~\ref{factor}) defines an $\underline{G}\times H$-equivariant structure on $R$ and hence (see \cite[14.21]{goe}) a descent for $R$ over $K$.

We have hence proven the following:
\begin{proposition}
Let $L/K$ be a finite normal modular extension of exponent less or equal than $e$ and let us denote $K[\bx]=K[X]/(X^{p^e})$. Let $R$ be an algebraic object over $L$, then the following are equivalent:
\begin{itemize}
\item[i)] there exists $R_0$ algebraic object over $K$ such that $R\simeq R_0\otimes_KL$ as $L$-algebraic objects;
\item[ii)] there exists a natural transformation of group valued functors 
\[\eta^R:\underline{\Aut}(L/K)\to\underline{\Aut}(R/K);\]
\item[iii)] there exists a group homomorphism 
\[\eta^R(K[\bx]):\underline{\Aut}(L/K)(K[\bx])\to\underline{\Aut}(R/K)(K[\bx])\]
that is a $K[\bx]$-linear group action of $\HG(L/K)$ on $R\otimes_KK[\bx]$ preserving the algebraic structure on $R$.
\end{itemize}
\end{proposition}

\begin{remark}
If $L/K$ is finite Galois, taking $e=0$ retrieves Galois descent.
\end{remark}

\section{A generalization}

Until now we have dealt with the action of $\HG(L/K)$ on the base change  to $L[\bx]$ of an $L$-algebraic object $R$. It can be useful to show that actually something slightly weaker is enough for $R$ to be defined over $K$, namely that it is enough for the $\HG(L/K)$-action to be defined over some infinitesimal deformation of $R$:
\begin{proposition}
Let $L/K$ be a finite normal modular extension of exponent less or equal than $1$ and let $V$ be a $K$-vector space, let $L[\bx]=L[X]/(X^p)$. Let $W$ be a sub-vector space of $V_L$, then $W\subset V_L$ is defined over $K$  if and only if there exists $\tilde{W}$ sub-module of $V_{L[\bx]}$ such that 
\begin{itemize}
\item[i)]  $\tilde{W}$ is free
\item[ii)] $\tilde{W}$ is invariant under the natural $\HG(L/K)$-action on $V_{L[\bx]}$
\item[iii)] $\tilde{W}\otimes_{L[\bx]}L=W$.
\end{itemize}
That is, $W\subset V_L$ is defined over $K$ if and only if it admits a  $\HG$-invariant deformation in $V_L$ over $L[\bx]$ which is free as $L[\bx]$-module.
\end{proposition}
\begin{proof}
One direction of the proposition follows easily from Theorem~\ref{de-svs}, simply taking $\tilde{W}$ to be $W\otimes_LL[\bx]$. Before proving the other direction we need one preparation lemma:
\begin{lemma}\label{prep}
Let $V$ be a $L$-vector space, let $L[\bx]=L[X]/(X^N)$ and let $W$ be a free sub-module of $V_{L[\bx]}=V\otimes_LL[\bx]$. Then $W$ is \emph{$\bx$-saturated}, by which we mean that if $v\cdot \bx^n\in W-\{0\}$ for some $v\in V_{L[\bx]}$  and $n<N$, then the image $\overline{v}$ of $v$ via the projection on $(\bx)$ is contained in $\overline{W}=W\otimes_{L[\bx]}L\subset V$.
\end{lemma}
\begin{proof}
Let us consider the decomposition $V_{L[\bx]}=V\cdot 1\oplus\dotsc\oplus V\cdot \bx^{N-1}$ and let us fix a basis $\{w_i\}_{i\in I}$ of $W$. Then for every $i\in I$
\[w_i=\sum_b v_{i,b}\bx^b\]
with $v_{i,b}\in V$.
As $W$ is free, the images of $w_i$  in $\overline{W}=W\otimes_{L[\bx]}L\subset V$ (that we can identify with $v_{i,0}$) form a basis of $\overline{W}$.
Let now $v\in V$ such that $v\otimes \bx^n$ is a non zero element of $W$. If $\bar{v}=0$ there is nothing to prove, otherwise let $\lambda_i\in L[\bx]$ such that
\[v\otimes\bx^n=\sum_iw_i\lambda_i.\]
Let us write $\lambda_i=\sum_a\lambda_{i,a}\bx^a$ with $\lambda_{i,a}\in L$, then
\[v\otimes\bx^n=\sum_j\bx^j\cdot\big[\sum_{a+b=j}\lambda_{i,a}v_{i,b}\big],\]
and in particular it follows that for every $j<n$, one has $\sum_{a+b=j}\lambda_{i,a}v_{i,b}=0$.
For $j=0$, this implies that $\sum_i \lambda_{i,0}v_{i,0}=0$, hence, as the $v_{i,0}$s are linearly independent, that $\lambda_{i,0}=0$ for every $i\in I$. But then by induction on $a$ the same argument shows that $\lambda_{i,a}=0$ for every $i\in I$ and $a<n$. In particular if we define 
\[\mu_i=\sum_a^{N-n-1}\lambda_{i,a+n}\bx^a\]
 it follows that $v'=v-\sum_i\mu_iw_i\in\bx\cdot L[\bx]$ and $v+v'\in W$, hence in particular $\overline{v+v'}=\overline{v}\in\overline{W}$.
\end{proof}

We are ready now to prove the reverse implication of the proposition, let hence $\tilde{W}$ be a free sub-module of $V_{L[\bx]}$ which is $\HG(L/K)$-invariant and such that its closed fiber $\tilde{W}\otimes_{L[\bx]}L$ equals $W$. In order to prove that $W$ is defined over $K$, by Theorem~\ref{hopf} it is enough to prove that $W$ is invariant under the natural $\GHA(L/K)$-action on $V_L$.

It is hence enough to check that $D_i^{(k)}(W)\subset W$ and that $D_g(W)\subset W$ for all $i,g$ as in the proof of Theorem~\ref{de-svs} and $k<p$. As $D_i^{(k)}=k!D_i^{(1)}$ for $k<p$, it is moreover enough to do the check for $k=1$.
 
Let us fix a basis $w_t$ of $W$ and a basis $v_s$ of $V$. Let us fix $w=w_t$ for some $t$ and write $w=\sum_s v_s\lambda_s$, with $\lambda_s\in L[\bx]$ and the usual abuse of notation $v_s$ for $v_s\otimes 1$. 

Let $\phi_i$ and $\phi_g\in\HG(L/K)$ be as in the proof of Theorem~\ref{de-svs}, then by hypothesis 
\[\phi_g(w)=\sum_s v_s \big[\sum_ag(\lambda_{s,a})\bx^a\big]\]
is in $W$, where we decompose $\lambda_s=\sum_a \lambda_{s,a}\bx^a$ with $\lambda_{s,a}\in L$. In particular $D_g(\bar{w})=\overline{\phi_g(w)}\in\overline{W}$ and hence $D_g(W)\subset W$. 
As for the $D_i^{(k)}$, we have that 
\[\begin{split}\phi_i(w)-w&=\sum_s v_s D_i^{(1)}(\lambda_s)\bx+\dotsc+\sum_sv_s D_i^{(k)}(\lambda_s)\bx^k+\dotsc\\
&=\bx\cdot\sum_s v_s D_i^{(1)}(\lambda_{s,0})+\bx^2\cdot \tilde{v}
\end{split}\]
is in $W$, with $\tilde{v}\in V_{L[\bx]}$.  In particular by Lemma~\ref{prep} $D_i^{(1)}(\bar{w})=\sum_s v_sD_i^{(1)}(\lambda_{s,0})\in \overline{W}$, hence $D_i^{(1)}(W)\subset W$ for every $i$.
\end{proof}

Note that even though the previous Proposition is proved only for exponent $1$, one can  use it as an induction step to get similar results for field extension of higher exponent.

\begin{remark}As in the case of Theorem~\ref{de-svs}, one can use the previous proposition to descend vector spaces endowed with additional algebraic structure, like ideals, Hopf-ideals or to descend morphisms between two objects defined over $K$, following the ideas of Section~\ref{sec:desc}.
\end{remark}

	\section{Open questions}

In order for the descent theory we described in this article to fully extend the classical Galois descent theory, there are still some open questions that need to be answered.

The first one concerns infinite (algebraic) extensions: if $L/K$ is Galois of infinite degree, then $G=\Aut(L/K)$ still provides a Galois correspondence by endowing $G$ with the Krull topology. If $X$ is a $K$-variety, and $\bar{K}$ is an algebraic closure of $K$, this allows for example to understand which sub-varieties of $X_{\bar{K}}=X\otimes_K \bar{K}$ are defined over $K$ just by looking at the natural action of the absolute Galois group of $K$ on the closed points of $X_{\bar{K}}$.
The Heerma--Galois group is defined and does work for any $L/K$ of finite exponent (even if the degree is infinite), thus leaving the following questions open:
\begin{question}
If $L/K$ is a modular normal field extension of infinite exponent, is it possible to define the Heerma--Galois group of such extension?
\end{question}

\begin{question}
Which topology should be then given on $\HG(L/K)$ in order to extend the Krull topology and to get descent for objects endowed with a continuous action of $\HG(L/K)$?
\end{question}

Another interesting point of view on descent theory comes from model theory. Namely, Galois correspondence has been has been generalized in this framework by Poizat \cite{poi} and then Pillay in \cite{pil} did the same for descent of constructible sets, working with a structure $\mathcal{M}$ (with some good properties) and its group of automorphisms fixing pointwise a subset $A$ of the universe $M$ of $\mathcal{M}$, thus obtaining descent for definable (that is, constructible) sets that are invariant under this group of automorphism (see \cite[Prop.~4.2]{pil}). This naturally leads to the following

\begin{question}
Is there a model theoretic proof of the descent results in this article? In which language should this be considered?
\end{question}

\addcontentsline{toc}{section}{\refname}
\printbibliography

\end{document}